\documentclass[10pt,reqno]{amsart}
\usepackage{amsmath}
\usepackage{amssymb}
\usepackage{amsthm}
\usepackage{eepic,epic}
\usepackage{epsfig}
\usepackage{graphicx}

\newlength{\defbaselineskip}
\newcommand{\setlinespacing}[1]%
           {\setlength{\baselineskip}{#1 \defbaselineskip}}

\numberwithin{equation}{section}

\newtheorem{thm}{Theorem}[section]
\newtheorem{prop}[thm]{Proposition}
\newtheorem{lem}[thm]{Lemma}
\newtheorem{cor}[thm]{Corollary}
\theoremstyle{definition}

\theoremstyle{remark}

\numberwithin{equation}{section}

\begin{document}

\title[Global unique continuation for the Schr\"odinger equation]
{Global unique continuation from a half space for the Schr\"odinger equation}

\author{Ihyeok Seo}

\thanks{2010 \textit{Mathematics Subject Classification.} Primary: 35B45, 35B60; Secondary: 35Q40, 42B35.}
\thanks{\textit{Key words and phrases.} Carleman estimate, Unique continuation,
Schr\"odinger equation, Fefferman-Phong class.}

\address{School of Mathematics, Korea Institute for Advanced Study, Seoul 130-722, Republic of Korea}
\email{ihseo@kias.re.kr}

\maketitle

\begin{abstract}
We obtain a global unique continuation result for the differential inequality
$|(i\partial_t+\Delta)u|\leq|V(x)u|$ in $\mathbb{R}^{n+1}$.
This is the first result on global unique continuation
for the Schr\"odinger equation with time-independent potentials $V(x)$ in $\mathbb{R}^{n}$.
Our method is based on a new type of Carleman estimates for the operator $i\partial_t+\Delta$ on $\mathbb{R}^{n+1}$.
As a corollary of the result, we also obtain a new unique continuation result for some parabolic equations.
\end{abstract}


\section{Introduction and statement of results}

In this paper we are concerned with Carleman estimates for the operator $i\partial_t+\Delta$
and their applications to global unique continuation for the Schr\"odinger equation
$i\partial_tu+\Delta u=V(x)u$ in $\mathbb{R}^{n+1}$.
More generally, we shall consider the differential inequality
\begin{equation}\label{Inequality}
|(i\partial_t+\Delta)u|\leq|V(x)u|,
\end{equation}
where $V$ is called a potential in $\mathbb{R}^n$.

Given a partial differential equation in $\mathbb{R}^n$, we say that it has
the (global) unique continuation from a non-empty open subset $\Omega\subset\mathbb{R}^n$
if its solution cannot vanish in $\Omega$ without being identically zero.
Historically, such property was studied in connection with the uniqueness of the Cauchy problem to which,
in many cases, it is equivalent.

The major method to attack unique continuation problems is based on so-called Carleman estimates
which are a type of weighted a priori estimates for the associated solutions.
The original idea goes back to Carleman~\cite{C}, who first introduced it to obtain
unique continuation for some second-order elliptic equations that need not have analytic or even smooth coefficients.
Since then, the method of Carleman estimates has played a central role in almost all subsequent developments.

Global unique continuation from a half space of $\mathbb{R}^{n+1}$ for the Schr\"odinger equation
with time-dependent potentials $V(x,t)$ has been studied by several authors (\cite{KS,S,LS,S2}).
In 1988, Kenig and Sogge ~\cite{KS} showed that solutions $u$ of the differential inequality
\begin{equation*}
|(i\partial_t+\Delta)u|\leq|V(x,t)u|
\end{equation*}
cannot vanish in a half space of $\mathbb{R}^{n+1}$
without being identically zero
if $V\in L^{\frac{n+2}2}(\mathbb{R}^{n+1})$.
This result was obtained by showing the following Carleman estimate:
\begin{equation}\label{KS-Carl}
\big\|e^{\beta\langle(x,t),\nu\rangle}u\big\|_{L^{\frac{2(n+2)}{n}}(\mathbb{R}^{n+1})}\leq
C\big\|e^{\beta\langle(x,t),\nu\rangle}(i\partial_t+\Delta)u\big\|_{L^{\frac{2(n+2)}{n+4}}(\mathbb{R}^{n+1})}
\end{equation}
with $C$, independent of $\beta\in\mathbb{R}$ and $\nu\in\mathbb{R}^{n+1}$,
whenever $u\in C_0^\infty(\mathbb{R}^{n+1})$.
Here $\langle\text{ },\text{ }\rangle$ denotes the usual inner product on $\mathbb{R}^{n+1}$.
Note that, when $\beta=0$, the estimate~\eqref{KS-Carl} is the inhomogeneous Strichartz estimate
for the Schr\"odinger equation (\cite{Str}).
In this regard, the later developments~\cite{IK,S,LS} have been made to extend ~\eqref{KS-Carl}
to mixed Lebesgue norms $L_t^qL_x^r$, with different exponents in space and time,
for which the inhomogeneous Strichartz estimate~\cite{KT,Fo,V} is known to hold.
More recently, these works were extended by the author ~\cite{S2} to more general mixed Wiener amalgam norms
$W(L^{q_1},L^{q_2})_tW(L^{r_1},L^{r_2})_x$
which in particular is the same as $L_t^qL_x^r$ when $q_1=q_2=q$ and $r_1=r_2=r$:
\begin{equation*}
\big\|e^{\beta\langle(x,t),\nu\rangle}u\big\|_{W(q_1,q_2)_tW(r_1,r_2)_x} \leq
C\big\|e^{\beta\langle(x,t),\nu\rangle}(i\partial_t+\Delta)u\big\|_
{W(\widetilde{q}_1',\widetilde{q}_2')_tW(\widetilde{r}_1',\widetilde{r}_2')_x},
\end{equation*}
where, for simplicity, we used $W(p,q)$ to denote $W(L^p,L^q)$.
Making use of such extended Carleman estimates, the above unique continuation result
of Kenig and Sogge was also improved to more general mixed spaces of potentials $V$
contained in $L_t^pL_x^s$ (\cite{S,LS}) and $W(L^{p_1},L^{p_2})_tW(L^{s_1},L^{s_2})_x$ (\cite{S2}).
There are also closely related unique continuation results (\cite{Z,B,KPV,IK,S})
for nonlinear Schr\"odinger equations of the form $(i\partial_t+\Delta)u=V(x,t)u+F(u)$,
where $F$ is a suitable nonlinear term.

In spite of these many works, there have been no results on global unique continuation
for ~\eqref{Inequality} with time-independent potentials $V(x)$.
The difficulty in this situation is that the conventional Carleman argument used to obtain unique continuation
from Carleman estimates heavily relies on H\"older's inequality.
For example, in~\eqref{KS-Carl}, one needs to control the norm of the right-hand side
by that of the other, as follows:
\begin{align*}
\big\|e^{\beta\langle(x,t),\nu\rangle}(i\partial_t+\Delta)u\big\|_{L^{\frac{2(n+2)}{n+4}}(\mathbb{R}^{n+1})}
&\leq
\big\|e^{\beta\langle(x,t),\nu\rangle}Vu\big\|_{L^{\frac{2(n+2)}{n+4}}(\mathbb{R}^{n+1})}\\
&\leq
\|V\|_{L^{\frac{n+2}2}(\mathbb{R}^{n+1})}
\big\|e^{\beta\langle(x,t),\nu\rangle}u\big\|_{L^{\frac{2(n+2)}{n}}(\mathbb{R}^{n+1})}.
\end{align*}
Note that the condition on the potential such as $V\in L^{(n+2)/2}(\mathbb{R}^{n+1})$ is determined
from using H\"older's inequality in such argument.
Hence it is not possible to obtain any unique continuation results for the time-independent case~\eqref{Inequality}
from a type of Carleman estimates used in time-dependent cases.

The main contribution of this paper is to establish the following new type of Carleman estimates
\begin{equation}\label{newtype}
\big\|e^{\beta\langle(x,t),\nu\rangle}u\big\|_{L_{t,x}^2(|V|)}\leq
C(V)\big\|e^{\beta\langle(x,t),\nu\rangle}(i\partial_t+\Delta)u\big\|_
{L_{t,x}^2(|V|^{-1})}
\end{equation}
which implies a global unique continuation result for the time-independent case~\eqref{Inequality}.
Here, $C(V)$ is a suitable constant depending on $V(x)$.
This new type makes it possible to use the Carleman argument
without the aid of H\"older's inequality. Indeed, note that
\begin{align*}
\big\|e^{\beta\langle(x,t),\nu\rangle}(i\partial_t+\Delta)u\big\|_{L_{t,x}^2(|V|^{-1})}
&\leq
\big\|e^{\beta\langle(x,t),\nu\rangle}Vu\big\|_{L_{t,x}^2(|V|^{-1})}\\
&=
\big\|e^{\beta\langle(x,t),\nu\rangle}u\big\|_{L_{t,x}^2(|V|)}.
\end{align*}

To obtain ~\eqref{newtype}, we consider the class of potentials used by Fefferman and Phong (\cite{F})
in the study of eigenvalue problems of the Schr\"odinger operator $-\Delta+V(x)$.
A function $V\in L_{\textrm{loc}}^p$, $1\leq p\leq n/2$, is said to be in
the so-called Fefferman-Phong class $\mathcal{F}^p$ if it satisfies that
$$\|V\|_{\mathcal{F}^p}=\sup_Q|Q|^{2/n}\bigg(\frac1{|Q|}\int_Q|V|^pdx\bigg)^{1/p}<\infty,$$
where the sup is taken over all cubes $Q$ in $\mathbb{R}^n$.
In particular, $L^{n/2}=\mathcal{F}^{n/2}$ and $L^{n/2,\infty}\subset \mathcal{F}^p$ for all $1\leq p<n/2$.
This class was also used in the study of unique continuation problems for the differential inequality
$|\Delta u|\leq|V(x)u|$ associated with the stationary Schr\"odinger equation (\cite{CS,W,CR,RV,RV3}).

The main result in this paper is the following Carleman estimate.

\begin{thm}\label{thm1}
Let\, $n\geq3$.
If\,\, $V\in\mathcal{F}^p$ for $p>(n-1)/2$,\,
then we have
\begin{equation}\label{Carl}
\big\|e^{\beta\langle(x,t),\nu\rangle}u\big\|_{L_{t,x}^2(|V|)} \leq
C\|V\|_{\mathcal{F}^p}\big\|e^{\beta\langle(x,t),\nu\rangle}(i\partial_t+\Delta)u\big\|_
{L_{t,x}^2(|V|^{-1})}
\end{equation}
with $C$, independent of $\beta\in\mathbb{R}$ and $\nu\in\mathbb{R}^{n+1}$,
whenever $u\in C_0^\infty(\mathbb{R}^{n+1})$.
\end{thm}

Making use of this estimate, we obtain a global unique continuation result from a half space
of $\mathbb{R}^{n+1}$ for the time-independent case
\begin{equation}\label{Inequality2}
|(i\partial_t+\Delta)u|\leq|V(x)u|.
\end{equation}
Let $H_{t}^{1}$ denote the usual Sobolev space of functions whose derivatives up to order $1$ with respect to time variable $t$ belong to $L^2$.
Similarly for $H_{x}^{2}$.

\begin{thm}\label{thm2}
Let $n\geq3$.
Suppose that $u\in H_{t}^{1}\cap H_{x}^{2}$ be a solution of ~\eqref{Inequality2}
which vanishes in a half space of $\mathbb{R}^{n+1}$.
Then it is identically zero if
$$V\in \mathcal{F}^p,\quad(n-1)/2<p\leq n/2,$$
with $\|V\|_{\mathcal{F}^p}<1/C$, where $C$ is given in ~\eqref{Carl}.
Additionally, if $p=n/2$, we assume that there is $\delta>0$ such that
$|V(x)|\geq \delta/R^2$ for $|x|\leq2R$ and all sufficiently large $R$.
\end{thm}

Now we would like to emphasize that Theorem~\ref{thm2} with a simple argument leads to
unique continuation for parabolic equations of the form
\begin{equation}\label{parabolic}
\partial_tu-\Delta u=\widetilde{V}(x,t)u,
\end{equation}
where $\widetilde{V}$ is a function in $\mathbb{R}^{n+1}$.
Some properties of unique continuation for ~\eqref{parabolic}
have been studied by several authors (e.g.,~\cite{E,EV}).
But, our result that follows directly from Theorem~\ref{thm2} is somehow a new and global one on the subject.
More precisely, let us consider the following parabolic equation
\begin{equation}\label{para}
\partial_tu-\Delta u=(g'(t)+t+V(x))u,
\end{equation}
where $g(t)$ is a differentiable function bounded below such that
$g(t)\rightarrow\infty$ as $|t|\rightarrow\infty$,
and $V$ satisfies the same conditions in Theorem~\ref{thm2}.
Then, we have the following result.

\begin{cor}\label{cor}
Let $n\geq3$ and $D$ be a half space of $\mathbb{R}^n$.
If $u\in H_{t}^{1}\cap H_{x}^{2}$ be a solution of ~\eqref{para}
which vanishes in a half space $D\times\mathbb{R}$ of $\mathbb{R}^{n+1}$,
then it is identically zero.
\end{cor}

The rest of this paper is organized as follows:
In Section~\ref{sec2} we deal with the class $A_p$ of weights and some estimates on weighted $L^2$ spaces
with weights in the Fefferman-Phong class,
which are to be used for the proof of Theorem~\ref{thm1}.
Then, using a localization argument in the phase space with Littlewood-Paley and multiplier theorems
on weighted $L^2$ spaces with weights in the class $A_2$,
we prove Theorem~\ref{thm1} in Section~\ref{sec3}.
Finally in Section~\ref{sec4}, we prove Theorem~\ref{thm2}
by making use of the Carleman estimate in Theorem~\ref{thm1}
and show how to deduce Corollary~\ref{cor} from Theorem~\ref{thm2}.

Throughout this paper, the letter $C$ stands for constants possibly different at each occurrence.
We also use the symbols $\widehat{f}$, $\mathcal{F}^{-1}(f)$
to denote the Fourier, the inverse Fourier transforms, respectively.


\section{Preliminaries}\label{sec2}

In this section we deal with some preliminary lemmas which will be used in the next section
for the proof of Theorem~\ref{thm1}.

\subsection{The class $A_p$ of weights}
Let $Q$ denote cubes in $\mathbb{R}^n$.
A weight\footnote{It is a nonnegative locally integrable function on $\mathbb{R}^n$.}
$w$ is said to be in the class $A_p$, $1<p<\infty$, if
$$\sup_Q\bigg(\frac1{|Q|}\int_Qw(x)dx\bigg)\bigg(\frac1{|Q|}\int_Qw(x)^{-\frac1{p-1}}dx\bigg)^{p-1}<\infty.$$
We also say that $w$ is in the class $A_1$ if, for almost all $x$,
$$M(w)(x)\leq Cw(x),$$
where $M(w)$ denotes the Hardy-Littlewood maximal function of $w$, given by
$$M(w)(x)=\sup_{Q\ni\,x}\frac1{|Q|}\int_Qw(y)dy.$$
Then, the followings are basic properties of $A_p$ weights:
\begin{equation}\label{properties}
w\in A_2\,\,\Leftrightarrow\,\, w^{-1}\in A_2\quad\text{and}\quad
A_p\subset A_q,\,\, 1\leq p<q<\infty.
\end{equation}
(For more details, see ~\cite{G,St2}.)

Let $\mu$ be a locally finite positive Borel measure on $\mathbb{R}^n$,
and let $\mu^\ast$ be the maximal function of $\mu$ defined by
$$\mu^\ast(x)=\sup_{Q\ni\,x}\{\mu(Q)/|Q|\}.$$
Then one has the following lemma which can be found in ~\cite{CoR}:

\begin{lem}\label{lem4}
Let\, $0<\delta<1$. If\, $\mu^\ast(x)<\infty$ for almost all $x$,
then $(\mu^\ast)^\delta$ is in the class $A_1$.
\end{lem}

\subsection{Some estimates on weighted $L^2$ spaces}

To begin with, let us consider the following initial value problem
associated to the free Schr\"odinger equation:
\begin{equation*}
\left\{
\begin{array}{ll}
i\partial_tu+\Delta u=0,\quad (x,t)\in\mathbb{R}^n\times\mathbb{R},\\
u(x,0)=f(x).
\end{array}\right.
\end{equation*}
It is well known that the solution can be written as
$$u(x,t)=e^{it\Delta}f(x)$$
using the Fourier transform,
where $e^{it\Delta}$ is called the free propagator which is given by
$$e^{it\Delta}f(x)=(2\pi)^{-n}\int_{\mathbb{R}^n} e^{i(x\cdot\xi-t|\xi|^2)}\widehat{f}(\xi)d\xi.$$

Now we give the following estimate on weighted $L^2$ spaces with respect to this propagator.

\begin{lem}\label{prop1}
Let $n\geq3$. If\, $V\in\mathcal{F}^p$ for $p>(n-1)/2$, then we have
\begin{equation}\label{dualhomo}
\bigg\|\int_{-\infty}^\infty e^{i(t-s)\Delta}F(s)ds\bigg\|_{L_{t,x}^2(|V|)}
\leq C\|V\|_{\mathcal{F}^p}\|F\|_{L_{t,x}^2(|V|^{-1})}.
\end{equation}
\end{lem}

\begin{proof}
It is enough to show that
\begin{equation}\label{homostr}
\big\|e^{it\Delta}f\big\|_{L_{t,x}^2(|V|)}\leq C\|V\|_{\mathcal{F}^p}^{1/2}\|f\|_2.
\end{equation}
Indeed, by duality this is equivalent to
\begin{equation}\label{dual}
\bigg\|\int_{-\infty}^\infty e^{-is\Delta}F(s)ds\bigg\|_2
\leq C\|V\|_{\mathcal{F}^p}^{1/2}\|F\|_{L_{t,x}^2(|V|^{-1})}.
\end{equation}
So, combining ~\eqref{homostr} and ~\eqref{dual}, we get ~\eqref{dualhomo}.

Now we show ~\eqref{homostr}.
We follow the argument in ~\cite{RV2} in which such estimate was implicitly
studied to treat regularizing properties of the propagator.
Using polar coordinates and changing variables $r^2=\lambda$, we see that
\begin{align*}
e^{it\Delta}f&=\int_0^\infty e^{-itr^2}\int_{S_r^{n-1}}e^{ix\cdot\xi}\widehat{f}(\xi)d\sigma_r(\xi)dr\\
&=\frac12\int_0^\infty e^{-it\lambda}\int_{S_{\sqrt{\lambda}}^{n-1}}
e^{ix\cdot\xi}\widehat{f}(\xi)d\sigma_{\sqrt{\lambda}}(\xi)\lambda^{-1/2}d\lambda.
\end{align*}
Hence by Plancherel's theorem, it follows that
\begin{align*}
\big\|e^{it\Delta}f\big\|_{L_{t,x}^2(|V|)}^2&\leq
C\int_{\mathbb{R}^n}\bigg(\int_0^\infty\bigg|\int_{S_{\sqrt{\lambda}}^{n-1}}
e^{ix\cdot\xi}\widehat{f}(\xi)d\sigma_{\sqrt{\lambda}}(\xi)\bigg|^2\lambda^{-1}d\lambda\bigg)|V(x)|dx\\
&\leq C\int_0^\infty\bigg(\int_{\mathbb{R}^n}
\bigg|\int_{S_r^{n-1}}e^{ix\cdot\xi}\widehat{f}(\xi)d\sigma_r(\xi)\bigg|^2|V(x)|dx\bigg)r^{-1}dr.
\end{align*}
Now, combining this and the known weighted restriction estimate\,\footnote{\,It can be found
in ~\cite{CS},~\cite{CR} and ~\cite{RV2}. For a simple proof, see also~\cite{RV}.}
$$\big\|\widehat{fd\sigma}\big\|_{L^2(|V|)}\leq C\|V\|_{\mathcal{F}^p}^{1/2}\|f\|_{L^2(S^{n-1})}$$
for $V\in\mathcal{F}^p$ with $p>(n-1)/2$, $n\geq3$, we get
\begin{align*}
\big\|e^{it\Delta}f\big\|_{L_{t,x}^2(|V|)}^2&\leq
C\|V\|_{\mathcal{F}^p}\int_0^\infty\int_{S_r^{n-1}}|\widehat{f}(\xi)|^2d\sigma_r(\xi)dr\\
&=C\|V\|_{\mathcal{F}^p}\|f\|_2^2
\end{align*}
as desired. This completes the proof.
\end{proof}

We close this section with recalling the following weighted $L^2$ estimate
for the resolvent of the Laplacian (see Lemma (2.10) in~\cite{CS}).

\begin{lem}\label{lem3}
Let $V\in\mathcal{F}^p$, $p>\frac{n-1}2$, $n\geq3$.
Suppose that $\text{Im }z\neq0$ and $\text{Re }z<0$.
Then one has
\begin{equation*}
\|f\|_{L^2(|V|)}\leq C\|V\|_{\mathcal{F}^p}\|(-\Delta+z)f\|_{L^2(|V|^{-1})}
\end{equation*}
with $C$ independent of $z\in\mathbb{C}$.
\end{lem}


\section{Carleman estimates on weighted $L^2$ spaces}\label{sec3}

Let us denote by $L(D)$ a first-order differential operator which is given by
$$L(D)=\langle \widetilde{c},\nabla\rangle+z,$$
where $\widetilde{c}\in \mathbb C^n$ and $z=a+ib\in \mathbb C$.
Then one can easily check that there is $L(D)$ such that
$$e^{\beta\langle(x,t),\nu\rangle}(i\partial_t+\Delta)e^{-\beta\langle(x,t),\nu\rangle}=i\partial_t+\Delta+L(D)$$
with $\widetilde{c},z$, depending on $\beta\in\mathbb{R}$ and $\nu\in\mathbb{R}^{n+1}$.
In fact,
$$L(D)=-2\beta\nabla\cdot\nu'+\beta^2|\nu'|^2-i\beta\nu_{n+1},$$
where $\nu=(\nu_1,...,\nu_n,\nu_{n+1})=(\nu',\nu_{n+1})$.
Hence the Carleman estimate~\eqref{Carl} in Theorem~\ref{thm1} follows directly from
the Sobolev-type inequality~\eqref{homo} below.
The aim of this section is to prove the following proposition.

\begin{prop}
Let\, $n\geq3$.
If\,\, $V\in\mathcal{F}^p$ for $p>(n-1)/2$,\,
then we have
\begin{equation}\label{homo}
\|u\|_{L_{t,x}^2(|V|)}\leq C\|V\|_{\mathcal{F}^p}
\|(i\partial_t+\Delta+L(D))u\|_{L_{t,x}^2(|V|^{-1})}
\end{equation}
with $C$, independent of $L(D)$, whenever $u\in C_0^\infty(\mathbb{R}^{n+1})$.
\end{prop}

\begin{proof}

Firstly, we reduce the proof to the following special case where $L(D)=\partial/\partial x_n$:
\begin{equation}\label{special}
\|u\|_{L_{t,x}^2(|V|)}\leq C\|V\|_{\mathcal{F}^p}
\|(i\partial_t+\Delta+\partial_{x_n})u\|_{L_{t,x}^2(|V|^{-1})}.
\end{equation}
Indeed, setting $u=e^{-i\langle d/2,x\rangle}v$, $d\in\mathbb{R}^n$,
we see that
$$(i\partial_t+\Delta+L(D))u=e^{-i\langle d/2,x\rangle}
(i\partial_t+\Delta+\langle \widetilde{c}-id,\nabla\rangle+z-i\langle \widetilde{c},d/2\rangle-|d/2|^2)v.$$
Hence we only need to consider the case where $\widetilde{c}\in\mathbb{R}^n$.
Similarly, by setting $u=e^{iat}v$, it follows that
$$(i\partial_t+\Delta+L(D))u=e^{iat}(i\partial_t+\Delta+\langle \widetilde{c},\nabla\rangle+ib)v.$$
So we may assume that $a=0$.
As it is well known, the Laplacian $\Delta$ is invariant under rotations.
From this fact, we can assume that
$\langle \widetilde{c},\nabla\rangle=c\partial_{x_n}$, where $c\in\mathbb{R}$.
So far, we have reduced the proof to the case where $L(D)=c\partial/\partial x_n+ib$
with $c,b\in\mathbb{R}$.
For simplicity of notation, we shall also assume that $c=1$ and $b=0$,
because it does not affect all the arguments in the proof.

Nextly, we consider the following one-dimensional maximal function\,\footnote{\,Namely,
$W=(M(|V|^\alpha))^{1/\alpha}$, where $M(f)$ denotes the Hardy-Littlewood maximal function of $f$.}
$$W(x)=\sup_\mu\bigg(\frac1{2\mu}\int_{x_n-\mu}^{x_n+\mu}|V(x_1,...,x_{n-1},\lambda)|^\alpha d\lambda\bigg)^{1/\alpha},\quad\alpha\geq1,$$
in the $x_n$ variable.
Then, if\, $V\in\mathcal{F}^p$ for $p>(n-1)/2$, and $1\leq \alpha<p$, one has
\begin{equation}\label{relation}
W\in\mathcal{F}^p\quad\text{and}\quad
\|W\|_{\mathcal{F}^p}\leq C\|V\|_{\mathcal{F}^p},
\end{equation}
which can be found in~\cite{CS}. (See Lemma (2.14) there.)
By ~\eqref{relation}, it is now enough to show the estimate~\eqref{special} by replacing $V$ with $W$:
\begin{equation}\label{repl}
\|u\|_{L_{t,x}^2(W)}\leq C\|W\|_{\mathcal{F}^p}
\|(i\partial_t+\Delta+\partial_{x_n})u\|_{L_{t,x}^2(W^{-1})}.
\end{equation}
The motivation behind this replacement is because $W\in A_1(\mathbb{R},dx_n)$ if $1<\alpha<p$.
This follows easily from ~\eqref{relation} and Lemma~\ref{lem4} with $\delta=1/\alpha$.
Hence, by ~\eqref{properties}, we also see that
$$W\in A_2(\mathbb{R},dx_n)\quad\text{and}\quad W^{-1}\in A_2(\mathbb{R},dx_n).$$
This makes it possible to deal with the estimate~\eqref{repl}
by making use of Littlewood-Paley and multiplier theorems on weighted $L^2$ spaces
with weights in the class $A_2$.
To do so, we shall work on Fourier transform side using a localization argument
in the phase space of the $\xi_n$ variable.

Note first that ~\eqref{repl} is equivalent to
\begin{equation}\label{multi}
\bigg\|\mathcal{F}^{-1}\bigg(\frac{\widehat{f}(\xi,\tau)}{\tau+|\xi|^2+i\xi_n}\bigg)\bigg\|_{L_{t,x}^2(W)}
\leq C\|W\|_{\mathcal{F}^p}\|f\|_{L_{t,x}^2(W^{-1})}
\end{equation}
for Schwartz functions $f\in\mathcal{S}(\mathbb{R}^{n+1})$.
Let us set
$$m(\tau,\xi)=(\tau+|\xi|^2+i\xi_n)^{-1}\quad\text{and}\quad
\widehat{Tf}(\tau,\xi)=m(\tau,\xi)\widehat{f}(\tau,\xi).$$
Also, let $\chi\in C_0^\infty(\mathbb{R})$ be such that
$\chi(r)=1$ if $|r|\sim1$ and zero otherwise.
Now we set $\chi_k(t)=\chi(2^kt)$ and
$$\widehat{T_kf}(\tau,\xi)=m(\tau,\xi)\chi_k(\xi_n)\widehat{f}(\tau,\xi)=m_k(\tau,\xi)\widehat{f}(\tau,\xi).$$
Since $W\in A_2(\mathbb{R},dx_n)$,
by the Littlewood-Paley theorem on weighted $L^2$ spaces (see Theorem 1 in~\cite{K}),
we see that
\begin{equation}\label{wrf}
\big\|\sum_kT_kf\big\|_{L_{t,x}^2(W)}\leq C\bigg(\iint\sum_k|T_kf|^2W dxdt\bigg)^{1/2}.
\end{equation}
Now we assume for the moment that
\begin{equation}\label{multi3}
\|T_kf\|_{L_{t,x}^2(W)}
\leq C\|W\|_{\mathcal{F}^p}\|f\|_{L_{t,x}^2(W^{-1})}
\end{equation}
with $C$, independent of $k\in\mathbb{Z}$.
Then, using this,
the right-hand side of ~\eqref{wrf} is bounded by
$$C\|W\|_{\mathcal{F}^p}\bigg(\iint\sum_k|f_k|^2W^{-1}dxdt\bigg)^{1/2},$$
where $\widehat{f_k}(\tau,\xi)=\chi_k(\xi_n)\widehat{f}(\tau,\xi)$.
Since $W^{-1}\in A_2(\mathbb{R},dx_n)$, by the Littlewood-Paley theorem again,
this is also bounded by
$$C\|W\|_{\mathcal{F}^p}\bigg(\iint|f|^2W^{-1}dxdt\bigg)^{1/2}$$
as desired.

It remains to show the estimate~\eqref{multi3}.
Namely, we have to show that ~\eqref{multi} holds with $C$, independent of $k\in\mathbb{Z}$,
for functions $f$ satisfying $\text{supp}\,\widehat{f}\subset\{(\xi,\tau):|\xi_n|\sim2^{-k}\}$.
To show this, we assume the following estimate which will be shown later:
\begin{equation}\label{multi2}
\bigg\|\mathcal{F}^{-1}\bigg(\frac{\widehat{f}(\xi,\tau)}{\tau+|\xi|^2+i2^{-k}}\bigg)\bigg\|_{L_{t,x}^2(W)}
\leq C\|W\|_{\mathcal{F}^p}\|f\|_{L_{t,x}^2(W^{-1})}
\end{equation}
whenever $\text{supp}\,\widehat{f}\subset\{(\xi,\tau):|\xi_n|\sim2^{-k}\}$.
Then, by taking differences we are reduced to showing that
\begin{equation*}
\bigg\|\mathcal{F}^{-1}\bigg(\frac{(2^{-k}-\xi_n)\widehat{f}(\xi,\tau)}
{(\tau+|\xi|^2+i\xi_n)(\tau+|\xi|^2+i2^{-k})}\bigg)\bigg\|_{L_{t,x}^2(W)}
\leq C\|W\|_{\mathcal{F}^p}\|f\|_{L_{t,x}^2(W^{-1})}.
\end{equation*}
By changing  variables $\tau+|\xi|^2\rightarrow\rho$ in the above,
we need to show that
\begin{equation}\label{multi4}
\bigg\|\int_{\mathbb{R}}e^{it\rho}\int_{\mathbb{R}}e^{i(t-s)\Delta}F(s)dsd\rho\bigg\|_{L_{t,x}^2(W)}
\leq C\|W\|_{\mathcal{F}^p}\|f\|_{L_{t,x}^2(W^{-1})},
\end{equation}
where
$$\widehat{F(\cdot,s)}(\xi)=\frac{2^{-k}-\xi_n}{(\rho+i\xi_n)(\rho+i2^{-k})}e^{-is\rho}\widehat{f(\cdot,s)}(\xi).$$
By Minkowski's inequality and the estimate ~\eqref{dualhomo} in Lemma~\ref{prop1},
the left-hand side of ~\eqref{multi4} is bounded by
\begin{equation}\label{dnff}
C\|W\|_{\mathcal{F}^p}\int_{\mathbb{R}}\bigg\|\mathcal{F}^{-1}
\bigg(\frac{2^{-k}-\xi_n}{(\rho+i\xi_n)(\rho+i2^{-k})}e^{-is\rho}\widehat{f(\cdot,s)}(\xi)\bigg)\bigg\|_
{L_{t,x}^2(W^{-1})}d\rho.
\end{equation}
Now, let us set
$$m_{k,\rho}(\xi_n)=\frac{2^{-k}-\xi_n}{(\rho+i\xi_n)(\rho+i2^{-k})}.$$
Then we see that
$$|m_{k,\rho}(\xi_n)|\leq \frac{C2^{-k}}{\rho^2+2^{-2k}}$$
and
$$\bigg|\frac{\partial m_{k,\rho}(\xi_n)}{\partial\xi_n}\bigg|\leq \frac{C}{\rho^2+2^{-2k}},$$
since we are assuming $\text{supp}\,\widehat{f}\subset\{(\xi,\tau):|\xi_n|\sim2^{-k}\}$.
Hence, $m_{k,\rho}(\xi_n)$ satisfies the conditions of the Marcinkiewicz multiplier theorem.
Then, by the multiplier theorem on weighted $L^2$ spaces (see Theorem 2 in~\cite{K}),
~\eqref{dnff} is also bounded by
\begin{equation*}
C\|W\|_{\mathcal{F}^p}\int_{\mathbb{R}}\frac{2^{-k}}{\rho^2+2^{-2k}}\|f\|_{L_{t,x}^2(W^{-1})}d\rho.
\end{equation*}
Hence we get ~\eqref{multi4}.
Now we show ~\eqref{multi2}. Note that
\begin{align*}
\bigg\|\iint e^{ix\cdot\xi+it\tau}
&\frac{\widehat{f}(\xi,\tau)}{\tau+|\xi|^2+i2^{-k}}d\xi d\tau\bigg\|_{L_{t,x}^2(W)}^2\\
&=\int W\int\bigg|\int e^{it\tau}\bigg(\int e^{ix\cdot\xi}
\frac{\widehat{f}(\xi,\tau)}{\tau+|\xi|^2+i2^{-k}}d\xi\bigg)d\tau\bigg|^2dtdx\\
&=\int W\int\bigg|\int e^{ix\cdot\xi}
\frac{\widehat{f}(\xi,\tau)}{\tau+|\xi|^2+i2^{-k}}d\xi\bigg|^2d\tau dx\\
&=\iint\bigg|\int e^{ix\cdot\xi}
\frac{\widehat{f}(\xi,\tau)}{\tau+|\xi|^2+i2^{-k}}d\xi\bigg|^2W dxd\tau.
\end{align*}
By Lemma~\ref{lem3}, the last term in the above is bounded by
$$\int\|W\|_{\mathcal{F}^p}^2\big\|\widehat{f(x,\cdot)}(\tau)\big\|_{L_{x}^2(W^{-1})}^2d\tau.$$
Now Plancherel's theorem in $\tau$ gives ~\eqref{multi2}.
This completes the proof.
\end{proof}


\section{Global unique continuation}\label{sec4}

In this section we first prove Theorem ~\ref{thm2} by making use of the Carleman estimate~\eqref{Carl},
and then show how to deduce Corollary~\ref{cor} from Theorem~\ref{thm2}.

\begin{proof}[Proof of Theorem ~\ref{thm2}]
Fix a unit vector $\nu_0\in\mathbb{R}^{n+1}$ arbitrarily. By translation we may assume that
$u$ vanishes in the half space $\{(x,t):\langle(x,t),\nu_0\rangle>0 \}$ of $\mathbb{R}^{n+1}$.
Obviously, by induction, it is enough to show that there is
$\sigma>0$ so that $u=0$ in the set
$$S_\sigma(\nu_0)=\{(x,t):-\sigma<\langle(x,t),\nu_0\rangle\leq0\}.$$

Let $\psi:\mathbb{R}^n\times\mathbb{R}\rightarrow[0,\infty)$ be a
smooth function which is supported in $\{(x,t):|x|,|t|\leq1\}$ and satisfies
$$\iint_{\mathbb{R}^n\times\mathbb{R}}\psi(x,t)dxdt=1.$$
For $0<\varepsilon<1$, we set
$\psi_\varepsilon(x,t)=\varepsilon^{-(n+1)}\psi(x/\varepsilon,t/\varepsilon)$.
Let $\phi:\mathbb{R}^n\times\mathbb{R}\rightarrow[0,1]$ be a smooth
function equal to $1$ in $\{(x,t):|x|,|t|\leq1\}$ and equal to $0$
in $\{(x,t):|x|\geq2 \text{ or } |t|\geq2\}$,
and for $R\geq1$ we put $\phi_R(x,t)=\phi(x/R,t/R)$.
Now, setting
$$\widetilde{u}(x,t)=(u\ast\psi_\varepsilon)(x,t)\phi_R(x,t),$$
we see that $\widetilde{u}\in C_0^\infty(\mathbb{R}^n\times\mathbb{R})$ is
supported in the set $H_\varepsilon(\nu_0)=\{(x,t):\langle(x,t),\nu_0\rangle\leq\varepsilon\}$.

Then, applying the Carleman estimate~\eqref{Carl} with $\nu=\nu_0$ to the function $\widetilde{u}$, we also see that
$$\big\|e^{\beta \langle(x,t),\nu_0\rangle}\widetilde{u}\big\|_{L_{t,x}^2(|V|;\,S_\sigma(\nu_0))}\leq
C\|V\|_{\mathcal{F}^p}\big\|e^{\beta\langle(x,t),\nu_0\rangle}(i\partial_t+\Delta)\widetilde{u}\big\|_
{L_{t,x}^2(|V|^{-1};\,H_\varepsilon(\nu_0))}.$$
Hence it follows by Fatou's lemma that
\begin{align}\label{468}
\nonumber\big\|e^{\beta \langle(x,t),\nu_0\rangle}&u\phi_R\big\|_{L_{t,x}^2(|V|;\,S_\sigma(\nu_0))}\\
&\leq
C\|V\|_{\mathcal{F}^p}\lim_{\varepsilon\rightarrow0}
\big\|e^{\beta\langle(x,t),\nu_0\rangle}(i\partial_t+\Delta)\widetilde{u}\big\|_
{L_{t,x}^2(|V|^{-1};\,H_\varepsilon(\nu_0))}.
\end{align}
Now, let us consider the term
\begin{align*}
\big\|e^{\beta\langle(x,t),\nu_0\rangle}(i\partial_t+&\Delta)\widetilde{u}\big\|_
{L_{t,x}^2(|V|^{-1};\,H_\varepsilon(\nu_0))}\\
=&\big\|e^{\beta\langle(x,t),\nu_0\rangle}((i\partial_t+\Delta)u\ast\psi_\varepsilon)\phi_R\big\|_
{L_{t,x}^2(|V|^{-1};\,H_\varepsilon(\nu_0))}\\
&+\big\|e^{\beta\langle(x,t),\nu_0\rangle}(u\ast\psi_\varepsilon)(i\partial_t+\Delta)\phi_R\big\|_
{L_{t,x}^2(|V|^{-1};\,H_\varepsilon(\nu_0))}\\
&+\|e^{\beta\langle(x,t),\nu_0\rangle}\nabla_x(u\ast\psi_\varepsilon)\cdot\nabla_x\phi_R\big\|_
{L_{t,x}^2(|V|^{-1};\,H_\varepsilon(\nu_0))}.
\end{align*}
Also, we note that
\begin{equation}\label{note2}
\|(f\ast g)\chi_{\{|x|,|t|\leq2R\}}\|_{L_{t,x}^2(|V|^{-1})}\leq CR^2\|f\|_{L_{t,x}^2}\|g\|_{L_{t,x}^1}
\end{equation}
if $|V|\geq C/|x|^2$ for $|x|\leq2R$.
When $p<n/2$, since $\delta/|x|^2\in\mathcal{F}^p$ and $|(i\partial_t+\Delta)u|\leq|Vu|\leq|(|V|+\delta/|x|^2)u|$,
we can assume that $|V|\geq\delta/|x|^2$.
Therefore, using ~\eqref{note2} and Lebesgue dominated convergence theorem,
we get
\begin{align*}
\lim_{\varepsilon\rightarrow0}\big\|e^{\beta\langle(x,t),\nu_0\rangle}(i\partial_t+&\Delta)\widetilde{u}\big\|_
{L_{t,x}^2(|V|^{-1};\,H_\varepsilon(\nu_0))}\\
=&\big\|e^{\beta\langle(x,t),\nu_0\rangle}(i\partial_t+\Delta)u\phi_R\big\|_
{L_{t,x}^2(|V|^{-1};\,H_0(\nu_0))}\\
&+\big\|e^{\beta\langle(x,t),\nu_0\rangle}u(i\partial_t+\Delta)\phi_R\big\|_
{L_{t,x}^2(|V|^{-1};\,H_0(\nu_0))}\\
&+\|e^{\beta\langle(x,t),\nu_0\rangle}\nabla_xu\cdot\nabla_x\phi_R\big\|_
{L_{t,x}^2(|V|^{-1};\,H_0(\nu_0))}
\end{align*}
because we are assuming $u\in H_t^1\cap H_x^2$.
Again by letting $R\rightarrow\infty$, we get
\begin{align*}
\lim_{R\rightarrow\infty}\lim_{\varepsilon\rightarrow0}
\big\|e^{\beta\langle(x,t),\nu_0\rangle}(i\partial_t&+\Delta)\widetilde{u}\big\|_
{L_{t,x}^2(|V|^{-1};\,H_\varepsilon(\nu_0))}\\
&=\big\|e^{\beta\langle(x,t),\nu_0\rangle}(i\partial_t+\Delta)u\big\|_
{L_{t,x}^2(|V|^{-1};\,H_0(\nu_0))}.
\end{align*}
Combining this, ~\eqref{468} and~\eqref{Inequality2}, we see that
\begin{align}\label{4567}
\nonumber\big\|e^{\beta \langle(x,t),\nu_0\rangle}u\big\|_{L_{t,x}^2(|V|;\,S_\sigma(\nu_0))}
&\leq C\|V\|_{\mathcal{F}^p}\big\|e^{\beta\langle(x,t),\nu_0\rangle}(i\partial_t+\Delta)u\big\|_
{L_{t,x}^2(|V|^{-1};\,H_0(\nu_0))}\\
\nonumber&\leq C\|V\|_{\mathcal{F}^p}\big\|e^{\beta\langle(x,t),\nu_0\rangle}Vu\big\|_
{L_{t,x}^2(|V|^{-1};\,H_0(\nu_0))}\\
&=C\|V\|_{\mathcal{F}^p}\big\|e^{\beta\langle(x,t),\nu_0\rangle}u\big\|_
{L_{t,x}^2(|V|;\,H_0(\nu_0))}.
\end{align}
Now we split the term
$\|e^{\beta\langle(x,t),\nu_0\rangle}u\|_{L_{t,x}^2(|V|;\,H_0(\nu_0))}$
into two parts
$$\big\|e^{\beta\langle(x,t),\nu_0\rangle}u\big\|_{L_{t,x}^2(|V|;\,S_\sigma(\nu_0))}
+\big\|e^{\beta\langle(x,t),\nu_0\rangle}u\big\|_
{L_{t,x}^2(|V|;\,\{(x,t):\langle(x,t),\nu_0\rangle\leq-\sigma\})}.$$
Since $C\|V\|_{\mathcal{F}^p}<1$, the first term can be absorbed into the left-hand side of
~\eqref{4567}.
Hence, we conclude that
\begin{align*}
\big\|e^{\beta (\langle(x,t),\nu_0\rangle+\sigma)}u\big\|_{L_{t,x}^2(|V|;\,S_\sigma(\nu_0))}
&\leq C\|u\|_{L_{t,x}^2(|V|;\,\{(x,t):\langle(x,t),\nu_0\rangle\leq-\sigma\})}\\
&\leq C\|\nabla_xu\|_{L_{t,x}^2},
\end{align*}
where we used the following Fefferman's inequality (\cite{F})
$$\|f\|_{L^2(|V|)}\leq C\|V\|_{\mathcal{F}^p}\|\nabla f\|_{L^2},\quad p>1.$$
Now, letting $\beta\rightarrow\infty$ implies that $u=0$ in the set $S_\sigma(\nu_0)$.
This completes the proof.
\end{proof}

\begin{proof}[Proof of Corollary~\ref{cor}]
Consider the following integral transformation
\begin{equation}\label{trans}
\widetilde{u}(x,s)=e^{-s^2/2}\int_{-\infty}^\infty e^{-its}e^{-g(t)}u(x,t)dt.
\end{equation}
Since $e^{-g(t)}$ is bounded\,\footnote{\,Recall that $g(t)$ is lower bounded.},
using the condition $u\in H_{t}^{1}\cap H_{x}^{2}$ and Plancherel's theorem,
we easily see that $\widetilde{u}\in H_{s}^{1}\cap H_{x}^{2}$.
Now, by differentiating ~\eqref{trans} it follows that
$$i\partial_s\widetilde{u}+\Delta\widetilde{u}=
e^{-s^2/2}\int_{-\infty}^\infty e^{-its}e^{-g(t)}[(-is+t)u+\Delta u]dt.$$
Note that integrating by parts gives
$$\int_{-\infty}^\infty e^{-its}(-is)e^{-g(t)}udt
=-\int_{-\infty}^\infty e^{-its}e^{-g(t)}(-g'(t)u+\partial_tu)dt$$
because $g(t)\rightarrow\infty$ as $|t|\rightarrow\infty$.
Hence, we see that
\begin{align*}
i\partial_s\widetilde{u}+\Delta\widetilde{u}
&=e^{-s^2/2}\int_{-\infty}^\infty e^{-its}e^{-g(t)}(g'(t)u-\partial_tu+tu+\Delta u)dt\\
&=-V(x)\widetilde{u}.
\end{align*}
Since $u$ vanishes in the half space $D\times\mathbb{R}$, so does $\widetilde{u}$.
Hence Theorem~\ref{thm2} directly implies that $\widetilde{u}\equiv0$.
Then, using Plancherel's theorem in ~\eqref{trans}, we conclude that $u\equiv0$.
This completes the proof.
\end{proof}



\end{document}